\documentclass{amsart}
\usepackage{stmaryrd}
\usepackage{mathrsfs}
\usepackage{amssymb}
\usepackage{amsmath}
\usepackage{titletoc}
\usepackage{extarrows}
\usepackage{varioref}
\usepackage{bm}
\usepackage{dsfont}
\usepackage[normalem]{ulem}
\usepackage[mathscr]{eucal}
\usepackage{tikz}
\usepackage{tikz-cd}
\usepackage[all]{xy}
\usepackage{color}
\usepackage{soul}
\usepackage{amscd}
\usepackage{amsfonts}
\usepackage{verbatim}
\usepackage[colorlinks,linkcolor=blue,citecolor=blue, pdfstartview=FitH]{hyperref}
\usepackage{graphicx}

\newcommand\ord{\operatorname{ord}}

\newcommand\Spec{\operatorname{Spec}}
\newcommand\rig{\operatorname{rig}}
\newcommand\Tr{\operatorname{Tr}}
\newcommand\Vol{\operatorname{Vol}}

\newcommand\as{\operatorname{as}}
\newcommand\an{\operatorname{an}}

\numberwithin{equation}{section}
\newtheorem{Definition}{Definition}[section]
\newtheorem{Theorem}[Definition]{Theorem}
\newtheorem{Lemma}[Definition]{Lemma}
\newtheorem{Proposition}[Definition]{Proposition}

\begin{document}
\begin{abstract}
    In this article, we prove a comparison theorem between 
    the Dwork cohomology introduced by Adolphson and Sperber and the rigid cohomology. As a corollary, we can calculate the rigid cohomology of Dwork isocrystal on torus. 
\end{abstract}
\title{Exponential sums and rigid cohomology}
\thanks{I would like to thank Hao Zhang and 
Professor Lei Fu for helpful discussions.}
\author{Peigen Li}
\address{Yau Mathematical Sciences Center, Tsinghua University, Beijing 100084, P. R. China}
\email{lpg16@mails.tsinghua.edu.cn}
\maketitle
\section{Introduction}
Let $p$ be a prime number, $\mathbf{Q}_p$ the field of $p$-adic numbers, $k=\mathbf{F}_{q}$ the finite field with $q=p^a$ elements. Let $k_i$ be the extension of $k$ of degree $i$ and let $\bar{k}$ be the algebraic closure of $k$. Fix a primitive $p$-th root of unity $\zeta_p$. Let $K_0$ be the unramified extension of $\mathbf{Q}_p(\zeta_p)$ of degree $a$. Let $\Omega$ be the completion of an algebraic closure of $K_0$. Denote by ``ord" the additive valuation on $\Omega$ normalized by $\ord(p)=1$. The norm on $\Omega$ is given by $|u|=p^{-\ord(u)}$ for any $u\in \Omega$.

For a morphism $f: X\rightarrow \mathbf{A}_k^1$ with $X$ being a $k$-scheme of finite type of dimension $n$, and the nontrivial additive character $\psi:k\rightarrow K_0^{\times}$ defined by $ \psi(t)=\zeta_p^{\Tr_{k/\mathbf{F}_p}(t)}$, 
define exponential sums
$$
S_i(X,f)=\sum_{x\in X(k_i)}\psi(\Tr_{k_i/k}(f(x))).
$$
The $L$-function is defined by 
$$L(X,f,t)=\exp\Big(\sum_{i=1}^{\infty}S_i(X,f)t^i/i\Big).$$ According to \cite[Theorem 6.3]{Etesse}, we have
$$
L(X,f,t)=\prod_{i=0}^{2n}\det(I-tF^{*}|H_{c,\rig}^i(X/K_0,f^{*}\mathcal{L}_{\psi}))^{(-1)^{i+1}},
$$
where $\mathcal{L}_{\psi}$ is the Dwork $F$-isocrystal defined over $\mathbf{A}_k^1$ associated to $\psi$ and $F^{*}$ is the Frobenius endomorphism on the space $H_{c,\rig}^i(X/K_0,f^{*}\mathcal{L}_{\psi}$).

In \cite[section 4]{baldassarri2004dwork}, Baldassarri and Berthelot compare the Dwork cohomology and the rigid cohomology for singular hypersurfaces. We prove a similar comparison theorem for the complex introduced by Adolphson and Sperber in \cite[section 2]{Ado1989} to study
the exponential sums on the torus $\mathbf{T}_k^n$. Suppose that $f$ is defined by a Laurent polynomial
$$
f(x_1,,\cdots, x_n)=\sum_{j=1}^Na_jx^{w_j}\in k[x_1, x_1^{-1},\cdots, x_n, x_n^{-1}].
$$
Let $\Delta(f)$ be the Newton polyhedron at $\infty$ of $f$ which is defined to be the convex hull in $\mathbf{R}^n$ of the set $\left\{w_j\right\}_{j=1}^N\cup \left\{(0,\cdots,0)\right\}$ and let $\delta$ be the convex cone generated by $\left\{w_j\right\}_{j=1}^N$ in $\mathbf{R}^n$. Let $\Vol(\Delta(f))$ be the volume of $\Delta(f)$ with respect to Lebesgue measure on $\mathbf{R}^n$. We say $f$ is \textit{nondegenerate with respect to $\Delta(f)$} if for any face $\sigma$ of $\Delta(f)$ not containing the origin, the Laurent polynomials $\frac{\partial f_{\sigma}}{\partial x_i}$, $i=1,\cdots,n$ have no common zero in $(\bar{k}^{\times})^n$, where $f_{\sigma}=\sum_{w_j\in \sigma}a_jx^{w_j}$. Define a weight function on $\delta\cap\mathbf{Z}^n$ by
$$
w(u):=\inf\left\{c:\;u\in c\Delta(f),\;c\geq 0\right\}.
$$
Note that there exists some $M\in \mathbf{Z}_{>0}$ such that $w(\delta\cap\mathbf{Z}^n)\subset \frac{1}{M}\mathbf{Z}_{\geq0}$.

Consider the Artin-Hasse exponential series: 
$$
E(t)=\exp\Big(\sum_{i=0}^{\infty}\frac{t^{p^i}}{p^i}\Big).
$$
By \cite[Lemma 4.1]{dwork1}, the series $\sum_{i=0}^{\infty}\frac{t^{p^i}}{p^i}$ has a zero at $\gamma\in \mathbf{Q}_p(\zeta_p)$ such that $\ord\gamma=1/(p-1)$ and $\zeta_p\equiv 1+\gamma\mod\gamma^2$. Set 
$$\theta(t)=E(\gamma t)=\sum_{i\geq 0}\lambda_it^i.$$ 
The series $\theta(t)$ is a splitting function in Dwork's terminology \cite[\S 4a]{dwork1}. In particular, we have $\ord \lambda_i\geq i/(p-1)$ and $\theta(1)=\zeta_p$. 

Fix an $M$-th root $\widetilde{\gamma}$ of $\gamma$ in $\Omega$. Let $K=K_0(\widetilde{\gamma})$, and $\mathcal{O}_K$ the ring of integers of $K$. Let $\hat{a}_j\in K$ be the Techm\"uller lifting of $a_j$ and set 
$$\hat{f}(x)=\sum_{j=1}^N\hat{a}_jx^{\omega_j}\in K[x_1,x_1^{-1},\cdots,x_n,x_n^{-1}].$$
For any $b>0,c\in\mathbf{R}$, consider the following spaces :
\begin{displaymath}
L(b,c)=\left\{\sum_{u\in \delta\cap\mathbf{Z}^n }a_ux^u:\; a_u\in K,\;\ord(a_u)\geq bw(u)+c\right\},
\end{displaymath}
$$
L(b)=\bigcup_{c\in\mathbf{R}} L(b,c), \quad L_{\delta}^{\dagger}=\bigcup_{b>0}L(b).
$$ 
Set $\gamma_l=\sum\limits_{i=0}^{l}\gamma^{p^i}/p^i,h(t)=\sum\limits_{l=0}^{\infty}\gamma_lt^{p^l}$. Then we have 
$$
\theta(t)=\exp\Big(\sum_{l=0}^{+\infty}\frac{\gamma^{p^l}t^{p^l}}{p^l}\Big)=\exp\Big(\gamma t+\sum_{l=1}^{+\infty}(\gamma_l-\gamma_{l-1})t^{p^l}\Big)
=\exp(h(t)-h(t^p)).
$$ Define 
$$
H(x)=\sum_{j=1}^Nh(\hat{a}_jx^{w_j}),~ F_0(x)=\exp(H(x)-H(x^q)).
$$ 
The estimate $\ord(\lambda_i)\geq \frac{i}{p-1}$ implies that $H(x)$ and $F_0(x)$ are well defined as formal Laurent series. In fact, we have 
$$
H(x)\in L\Big(\frac{1}{p-1},0\Big),\;F_0(x)\in L\Big(\frac{p}{q(p-1)},0\Big).
$$
Define an operator $\psi_q$ on formal Laurent series by 
$$
\psi_q\Big(\sum_{u\in \mathbf{Z}^n}a_ux^u\Big)=\sum_{u\in\mathbf{Z}^n}a_{qu}x^u.$$
Let $\alpha=\psi_q\circ F_0$.   Formally, we have
$$
\alpha=\psi_q\circ \exp(H(x)-H(x^q))=\exp(-H(x))\circ \psi_q \circ \exp(H(x)).
$$
For $i=1,\cdots,n$, define operators $$E_i=x_i\partial/\partial x_i,~H_i=E_iH,~\hat{D}_i=E_i+H_i=E_i+E_iH.$$
Formally, we have
$$
\hat{D}_i=\exp(-H(x))\circ E_i \circ \exp(H(x))
$$
for $i=1,\cdots,n$.
Note that $\alpha$ and $\hat{D}_i$ operate on $L_{\delta}^{\dagger}$. One can show that $\hat{D}_i$ commute with one another. Denote by $K_{\cdot}(L_{\delta}^{\dagger},\underline{\hat{D}})$ the Koszul complex on $L^{\dagger}_{\delta}$ associated to $\hat{D}_1,\cdots,\hat{D}_n$. We have $\alpha\circ \hat{D}_i=q\hat{D}_i\circ \alpha$ for all $i$. This implies that $\alpha$ induces a chain map on $K_{\cdot}(L_{\delta}^{\dagger},\underline{\hat{D}})$. Now we can give the main theorem in the present paper.
\begin{Theorem}\label{Thm1} ${}$
We have an isomorphism
$$
H_{n-i}(K_{\cdot}(L_{\delta}^{\dagger},\underline{\hat{D}}))\xrightarrow{\sim}H^{i}_{\rig}(\mathbf{T}^n_k/K,f^{*}\mathcal{L}_{-\psi})
$$
for each $i$. The endomorphism $\alpha$ on the left corresponds to the endomorphism $F_{*}=q^n(F^{*})^{-1}$ on the rigid cohomology. Furthermore, $\alpha$ is bijective on each homology group of $K_{\cdot}(L^{\dagger}_{\delta},\underline{\hat{D}})$.
\end{Theorem}
As a corollary, we prove the following theorem, which eliminates the condition $p\neq 2$ in \cite{Bourgeois}. 
\begin{Theorem}\label{Thm2} ${}$
Suppose that $f:\mathbf{T}_k^n\rightarrow \mathbf{A}_k^1$ is nondegenerate with respect to $\Delta(f)$ and that $\dim\Delta(f)=n$. Then

(i) $H_{c,\rig}^i(\mathbf{T}_k^n,f^{*}\mathcal{L}_{\psi})=H_{\rig}^i(\mathbf{T}_k^n,f^{*}\mathcal{L}_{-\psi})=0$ if $i\neq n$.

(ii) $\dim H_{c,\rig}^n(\mathbf{T}_{k}^n,f^{*}\mathcal{L}_{\psi})=\dim H_{\rig}^n(\mathbf{T}_k^n,f^{*}\mathcal{L}_{-\psi})=n!\Vol(\Delta(f))$.
\end{Theorem}

\section{Comparison Theorem}
 By \cite[Lemma 4.1]{dwork1}, there exists an element $\pi\in \mathbf{Q}_p(\zeta_p)$ such that $\pi^{p-1}+p=0$ and
$$\zeta_p\equiv 1+\pi\mod\pi^2.$$ Hence $\ord(\pi-\gamma)\geq \frac{2}{p-1}$ as $\zeta_p-1\equiv\pi\equiv\gamma\mod(\zeta_p-1)^2$.
By \cite[Lemma 4.1]{dwork1} and the definition of the weight function, we have $$G(x):=\exp(\pi(\hat{f}(x)-\hat{f}(x^q)))\in  L\Big(\frac{p-1}{pq},0\Big).$$ For $i=1,\cdots,n$, define operators
$$
D_i=\exp(-\pi\hat{f}(x))\circ E_i \circ \exp(\pi \hat{f}(x)).
$$ 
We have the Koszul complex $K_{\cdot}(L_{\delta}^{\dagger},\underline{D})$ and its endomorphism induced by $\alpha_1=\exp(-\pi\hat{f}(x))\circ \psi_q\circ \exp(\pi\hat{f}(x))$. Let $R=\exp(H(x)-\pi\hat{f}(x))$. We have 
\begin{equation}\label{commute-R}
  \alpha_1=R\circ\alpha \circ R^{-1},~D_i=R\circ \hat{D}_i \circ R^{-1}
\end{equation} for all $i$.
\begin{Proposition}\label{Compa-D}
(i) $R,R^{-1}\in L(b_0)$ where $b_0=\min\left\{\frac{1}{p-1},\frac{p-1}{p}\right\}$.

(ii) The multiplication by $R$ defines an isomorphism of Koszul complexes $\beta:K_{\cdot}(L_{\delta}^{\dagger},\underline{\hat{D}})\xrightarrow{\cdot R}  K_{\cdot}(L_{\delta}^{\dagger},\underline{D})$ and $\alpha_1\circ \beta=\beta\circ\alpha$.
\end{Proposition}
\begin{proof}
(ii) follows from (i) and (\ref{commute-R}). Let $h_1(t)=\sum_{l=1}^{+\infty}\gamma_lt^{p^l}$. Then
$$
R=\exp((\gamma-\pi)\hat{f}(x))\prod_{j=1}^N\exp(h_1(\hat{a}_jx^{w_j})).
$$The estimate $\ord(\gamma-\pi)\geq \frac{2}{p-1}$ implies that 
$$
\exp((\gamma-\pi)\hat{f}(x))=\prod_{j=1}^N\exp((\gamma-\pi)\hat{a}_jx^{w_j})\in L\Big(\frac{1}{p-1},0\Big).
$$
 Note that $\exp(\gamma_lt^{p^l})=\sum_{k=0}^{+\infty}\frac{\gamma_l^{k}}{k!}t^{kp^l}$ and 
 $$
 \ord(\gamma_l)=\ord\Big(-\sum_{i=l+1}^{+\infty}\gamma^{p^i}/p^i\Big)\geq \frac{p^{l+1}}{p-1}-l-1.
 $$ Then
 $$
 \ord\Big(\frac{\gamma_l^k}{k!}\Big)\geq k\Big(\frac{p^{l+1}}{p-1}-l-1-\frac{1}{p-1}\Big)\geq \frac{p-1}{p}\cdot kp^l
 $$ for each $l\geq 1$. 
As the term $c_it^i$ in the expansion for $\exp(h_1(t))=\prod_{l=1}^{+\infty}\exp(\gamma_lt^{p^l})$ only depends on finitely many $\exp(\gamma_lt^{p^l})$, we have 
$$\ord(c_i)\geq \frac{p-1}{p},~\textrm{and}~
\prod_{j=1}^N\exp(h_1(\hat{a}_jx^{w_j}))\in L\Big(\frac{p-1}{p},0\Big).
$$
Hence $R\in L\Big(\frac{1}{p-1}\Big)L\Big(\frac{p-1}{p}\Big)\subset L(b_0)$. The same is true for $R^{-1}$.
\end{proof}
  Consider the twisted de Rham complex $C^\cdot(L_{\delta}^{\dagger},\underline{D})$ defined by
  \begin{displaymath}
  C^k(L_{\delta}^{\dagger},\underline{D})=\bigoplus_{1\leq i_1<\cdots<i_k\leq n}L_{\delta}^{\dagger}\frac{dx_{i_1}}{x_{i_1}}\wedge\cdots\wedge \frac{dx_{i_k}}{x_{i_k}}\cong (L_{\delta}^{\dagger})^{\binom{n}{k}}
  \end{displaymath}
  with the differential $d:C^k(L_{\delta}^{\dagger},\underline{D})\rightarrow C^{k+1}(L_{\delta}^{\dagger},\underline{D})$ given by 
  $$
  d(\omega)=\exp(-\pi\hat{f})\circ d_x\circ \exp(\pi\hat{f})(\omega)
  $$
  for any $
  \omega\in C^k(L_{\delta}^{\dagger},\underline{D})
  $.
  Define an endomorphism $\alpha^{(k)}_1$ of $C^k(L_{\delta}^{\dagger},\underline{D})$ by
  $$
  \alpha^{(k)}_1=\bigoplus_{1\leq i_1<\cdots<i_k\leq n}q^{n-k}\alpha_1.
  $$
  The equalities $\alpha_1\circ D_i=qD_i\circ\alpha_1$ imply that $\alpha_1$ induces a chain map on $C^{\cdot}(L_{\delta}^{\dagger},\underline{D})$. For each $0\leq k \leq n$, we have a natural isomorphism
  \begin{equation}\label{Commu-Dwork1}
  H^{k}(C^\cdot(L_{\delta}^{\dagger},\underline{D}))\simeq H_{n-k}(K_{\cdot}(L_{\delta}^{\dagger},\underline{D}))
  \end{equation} 
  which is compatible with their endomorphisms $\alpha_1$.

Let $\Sigma$ be the fan of all faces of the dual cone of $\delta$. Let $X$ be the toric scheme over $R=\mathcal{O}_K$ associated to $\Sigma$. By the definition of $\Sigma$, we have $X=\Spec(A)$ with $A=R[x^{\delta\cap \mathbf{Z}^n}]$. Set 
$$
\mathcal{S}=\left\{\sum_{j=1}^{N}r_jw_j:\;0\leq r_1,\cdots,r_N\leq 1,\;\sum_{j=1}^{N}r_jw_j\in\mathbf{Z}^{n}\right\}.
$$ 
Since $\mathcal{S}$ is discrete and bounded, it is finite. Set $\mathcal{S}=\left\{s_1,\cdots,s_L\right\}$. Then $A$ is generated by $\left\{x^{s_i}|s_i\in \mathcal{S}\right\}$ as an $R$-algebra. More precisely, we have
$ A\simeq R[y_1,y_2,\cdots,y_L]/I$ by \cite[Proposition 1.1.9]{cox}, where
$$
I=\left\{\sum_{a,b}\Big(y^{a_1}_1\cdots y^{a_L}_L-y^{b_1}_1\cdots y^{b_L}_L\Big):\;\sum_{i=1}^La_i s_i=\sum_{i=1}^Lb_i s_i,\;a_i,\;b_i\in \mathbf{Z}_{\geq 0}\right\}.
$$

Consider the canonical immersions $X\rightarrow \mathbf{A}_R^L\rightarrow \mathbf{P}_R^L$. Let $\overline{X}$ be the closure of $X$ in $\mathbf{P}_R^L$. Let $\widehat{\overline{X}}$ be the formal completion of $\overline{X}$ with respect to the $p$-adic topology. Denote by $X_k,X_K$ the special fiber and the generic fiber of $X$ over $R$, respectively. Denote by $X^{\an}_K$ the analytic space associated to $X_K$. By \cite[1.2.4 (ii)]{berthelot1996cohomologie}, $X^{\an}_K$ is a strict neighborhood of $]X_k[_{\widehat{\overline{X}}}$ in $]\overline{X}_k[_{\widehat{\overline{X}}}$. Let $V$ be a strict neighborhood of $]X_k[_{\widehat{\overline{X}}}$ in $]\overline{X}_k[_{\widehat{\overline{X}}}$ and $E$ a sheaf of $\mathcal{O}_V$-modules. Define 
$$
j^{\dagger}E:=\varinjlim_{V'}j_{V'*}j^{*}_{V'}E,
$$ where the inclusion $j_{V'}:V'\rightarrow V$ runs through the strict neighborhoods of $]X_k[_{\widehat{\overline{X}}}$ contained in $V$.
\begin{Lemma}\label{Compa-Basic}
We have an isomorphism
$
L_{\delta}^{\dagger}\simeq \Gamma(X^{\an}_K,j^{\dagger}\mathcal{O}_{X^{\an}_K})
$
and $H^{i}(X^{\an}_K,j^{\dagger}(\mathcal{O}_{X^{\an}_K}))=0$ for all $i\geq 1$. 
\end{Lemma}
\begin{proof}
By \cite[Exemples 1.2.4(iii)]{berthelot1996cohomologie}, $V_\lambda:=B^{L}(0,\lambda)\cap X_K^{\an}$ form a cofinal system of strict neighborhoods of $]X_k[_{\widehat{\overline{X}}}$ in $X_K^{\an}$ when $\lambda\rightarrow{1^{+}}$, where $B^L(0,\lambda)$ is the closed polydisc with radius $\lambda$ of dimension $L$ in $\mathbf{A}^{L,\an}_K$. 
Let $j_{\lambda}$ be the inclusion $V_{\lambda}\rightarrow X_{K}^{\an}$ and fix a $\lambda_0$. By \cite[1.2.5]{berthelot1997finitude}, we have 
$$
j^{\dagger}\mathcal{O}_{X_{K}^{\an}}
\simeq j_{\lambda_0*}j^{\dagger}\mathcal{O}_{V_{\lambda_0}}\simeq Rj_{\lambda_0*}j^{\dagger}\mathcal{O}_{V_{\lambda_0}}.
$$ Note that the inclusion $V_\lambda\rightarrow V_{\lambda_0}$ is affine and $V_{\lambda}$ is affinoid for each $\lambda\leq \lambda_0$. We have 
$$H^i(X_{K}^{\an},j^{\dagger}\mathcal{O}_{X_{K}^{\an}})\simeq H^{i}(V_{\lambda_0},j^{\dagger}\mathcal{O}_{V_{\lambda_0}})\simeq \varinjlim_{\lambda<\lambda_0}H^{i}(V_\lambda,\mathcal{O}_{V_{\lambda}}).$$ Hence the second assertion holds.
Set
\begin{displaymath}
W_L=\bigcup_{\lambda>1}\left\{\sum_{v\in \mathbf{Z}_{\geq 0}^L}a_vy^v:\;a_v\in K,\;|a_v|\lambda^{|v|}\rightarrow 0~\as~|v|\rightarrow +\infty\right\}.
\end{displaymath}
We have
$
\Gamma(X_{K}^{\an},j^{\dagger}\mathcal{O}_{X_{K}^{\an}})=\varinjlim_{\lambda>1}\Gamma(V_\lambda,\mathcal{O}_{V_{\lambda}})\cong W_L/IW_L.
$
For any $u\in \mathbf{Z}^n\cap \delta$, denote the set of solutions $v_1s_1+\cdots+v_Ls_L=u$ in $\mathbf{Z}^L_{\geq0}$ by $S_u$. Define a morphism $\chi:W_L\rightarrow L^{\dagger}_{\delta}$ by 
$$
\chi\Big(\sum_{v\in\mathbf{Z}_{\geq0}^{L}}a_vy^{v}\Big)=\sum_{u\in\delta\cap \mathbf{Z}^n}\Big(\sum_{v\in S_u}a_v\Big)x^{u}.$$ 
Since the map 
$
v\mapsto u=v_1s_1+\cdots+v_Ls_L
$ is a linear continuous map, there exists some $B>0$ such that $|u|\leq B|v|$. We have 
$$
\Big|\sum_{v\in S_u}a_v\Big|\lambda^{B^{-1}|u|}\leq \max_{v\in S_u}|a_v|\lambda^{|v|}
$$
for any $\lambda>1$. Thus $\chi$ is well defined and it induces a morphism 
$
\bar{\chi}:W_L/IW_L\rightarrow L_{\delta}^{\dagger}.
$
It suffices to show $\bar{\chi}$ is an isomorphism. 

Take
$
g(x)=\sum_{u\in\delta\cap \mathbf{Z}^n}b_ux^u\in L(b,c)
$ for some $b>0,c\in \mathbf{R}$.
By the definition of $\mathcal{S}$, for any nonzero element $u\in\delta\cap\mathbf{Z}^n$, we can find a subset $s(u)\subset \left\{1,\cdots,L\right\}$ such that 
$$ 
u=\sum_{i\in s(u)}v_{i}s_{i},\quad
 v_i\in \mathbf{Z}_{\geq 1}
$$ and $w(s_{i})^{-1}s_{i},w(u)^{-1}u$ lie in the same codimension 1 face of $\Delta(f)$ for any $i\in s(u)$. By \cite[Lemma 1.9 (c)]{Ado1989} and the choice of $M$, we have 
 $$w(u)=\sum_{i\in s(u)}v_{i}w(s_{i})\geq \frac{\sum_{i\in s(u)}|v_i|}{M}.$$
Set $y^{v(u)}=\prod_{i\in s(u)}y^{v_{i}}_{i}$ and
$
h(y)=\sum_{u}b_uy^{v(u)}.
$
We have $$\ord(b_u)\geq bw(u)+c\geq \frac{b}{M}|v(u)|+c.$$ 
So $h(y)\in W_L$ and $\chi(h(y))=g(x)$. Hence $\bar{\chi}$ is surjective.

Take $f=\sum_{v}a_vy^v\in W_L$ such that $\ord(a_v)\geq b|v|+c$ for some $b,c\in \mathbf{R}$ and $b>0$.
If $\chi(f)=0$, we have
$$\sum_{v\in S_u}a_v=0$$ for any $u$. For each $u\in \delta\cap\mathbf{Z}^n$, take some $v_u \in S_u$ such that $|v_u|=\min\left\{|v||v\in S_u\right\}$. By Lemma \ref{finite-contr} below, we can find $h_1,\cdots,h_m\in I$ such that for any $v\in S_u$, there exist $g_{vu1},\cdots, g_{vun}\in R[y_1,\cdots,y_L]$ such that
$$
y^v-y^{v_u}=\sum_{j=1}^mg_{vuj}h_j
$$ and $\deg(g_{vuj})\leq |v|$ for all $j$. So 
$$f=\sum_{u}\sum_{v\in S_u}a_vy^v=\sum_{u}\sum_{v\in S_u}a_v(y^v-y^{v_u})=\sum_{j=1}^mh_j\Big(\sum_{u}\sum_{v\in S_u}a_vg_{vuj}\Big).$$
Set $g_{vuj}=\sum_{k}g^k_{vuj}y^k$ with $|k|\leq |v|$ and $\ord(g^k_{vuj})\geq 0$. Define
$$
g_j:=\sum_{u}\sum_{v\in S_u}a_vg_{vuj}=\sum_{k}y^k\Big(\sum_{u}\sum_{v\in S_u}a_vg^k_{vuj}\Big).
$$ We have 
$$
\ord\Big(\sum_{u}\sum_{v\in S_u}a_vg^k_{vuj}\Big)\geq \min_{v\in S_u,u}\ord(a_vg^k_{vuj})\geq \min_{v\in S_u,u}\ord(a_v)\geq \min_{v\in S_u,u}b|v|+c\geq b|k|+c.
$$ Hence $g_j\in W_L$ and $f\in IW_L$ which imply that $\bar{\chi}$ is injective.
\end{proof}
\begin{Lemma}\label{inf-finite}
  For any subset $J$ of $\mathbf{Z}_{\geq 0}^n$, there exists a finite subset $J_0$ of $J$ such that $J\subset \bigcup_{v\in J_0}(v+\mathbf{Z}_{\geq 0}^n)$.
\end{Lemma}
\begin{proof}
      We use induction on $n$. When $n=1$, it is trivial. Suppose that our assertion holds for any subset of $\mathbf{Z}_{\geq 0}^n$. Let $J$ be a subset of $\mathbf{Z}_{\geq 0}^{n+1}$. Take 
      $$
      J'=\left\{v\in \mathbf{Z}_{\geq 0}^{n}:\;(a,v)\in J~\textrm{for some $a\in \mathbf{Z}_{\geq 0}$}\right\}.
      $$By the induction hypothesis, there exists a finite subset $J_0'=\left\{v'_1,\cdots,v'_r\right\}$ such that 
      $$
      J'\subset \bigcup_{i=1}^r\Big(v'_i+\mathbf{Z}_{\geq 0}^{n}\Big).
      $$
      For each $i$, take some $a_i\in \mathbf{Z}_{\geq 0}$ such that $v_i=(a_i,v'_i)\in J$. For any $b\in \mathbf{Z}_{\geq 0}$, let 
      $$
      J_b=\left\{(c_1,\cdots,c_{n+1})\in J:\;c_1=b\right\}.
      $$Using the induction hypothesis again, there exists a finite subset $J_{b,0}$ such that 
      $$
      J_b\subset \bigcup_{v\in J_{b,0}}\Big(v+\mathbf{Z}_{\geq 0}^{n+1}\Big).
      $$Set $a_0=\max_{1\leq i\leq r}\left\{a_i\right\}$. We have 
      \begin{eqnarray*}
      J&\subset& \Big(\bigcup_{b< a_0}J_b\Big)\bigcup \Big(\bigcup_{i=1}^r(v_i+\mathbf{Z}_{\geq 0}^{n+1})\Big)\\
      &\subset &\Big(\bigcup_{b< a_0}\bigcup_{v\in J_{b,0}}(v+\mathbf{Z}_{\geq 0}^{n+1})\Big)\bigcup \Big(\bigcup_{i=1}^r(v_i+\mathbf{Z}_{\geq 0}^{n+1})\Big)
      \end{eqnarray*}
      We can take $J_0=\Big(\bigcup_{b<a_0}J_{b,0}\Big)\bigcup \left\{v_1,\cdots,v_r\right\}$.
\end{proof}
\begin{Lemma}\label{finite-contr}
   There are finitely many elements $h_1,\cdots,h_m\in I$ such that for any $y^a-y^b\in I$, there exist $g_1,\cdots,g_m\in R[y_1,\cdots,y_L]$ such that 
  $$
y^a-y^b=\sum_{i=1}^mg_ih_i
$$ and $\deg(g_i)\leq \max\left\{|a|,|b|\right\}$.
\end{Lemma}
\begin{proof}
  Take $J=\left\{(a,b):\;y^a-y^b\in I\right\}-\left\{(0,0)\right\}$. By Lemma \ref{inf-finite} above, there is a finite subset $J_0=\left\{(a_1,b_1),\cdots,(a_m,b_m)\right\}\subset J$ such that $J\subset \bigcup_{v\in J_0}(v+\mathbf{Z}^{2L}_{\geq 0})$. Take $h_i=y^{a_i}-y^{b_i}$. We prove the assertion by induction on $|a|+|b|$. When $|a|+|b|=0$, the assertion is trivial. Take $(a,b)\in J$. Suppose that the assertion holds for any $y^{a'}-y^{b'}\in J$ with $|a'|+|b'|<|a|+|b|$. By the choice of $J_0$, there is some $i_0$ such that $(a,b)\geq (a_{i_0},b_{i_0})$. We have $|a-a_{i_0}|+|b-b_{i_0}|<|a|+|b|$. By induction hypothesis, there exist $g_1',\cdots,g_m'\in R[y_1,\cdots,y_L]$ such that 
  $$
  y^{a-a_{i_0}}-y^{b-b_{i_0}}=\sum_{i=1}^mg_i'h_i
  $$ and $\deg(g_i')\leq \max\left\{|a-a_{i_0}|,|b-b_{i_0}|\right\}$.
  
  If $|a_{i_0}|\geq |b_{i_0}|$, we have 
  $$
  y^a-y^b=y^{a-a_{i_0}}(y^{a_{i_0}}-y^{b_{i_0}})+y^{b_{i_0}}(y^{a-a_{i_0}}-y^{b-b_{i_0}}).
  $$Take $g_{i_0}=y^{a-a_{i_0}}+y^{b_{i_0}}g_{i_0}'$ and $g_i=y^{b_{i_0}}g_i'$ for any $i\neq i_0$. We have 
  $$
y^a-y^b=\sum_{i=1}^mg_ih_i
$$ and $\deg(g_i)\leq \max\{|a-a_{i_0}|,|b_{i_0}|+\deg(g_i')\}\leq \max\{|a|,|b|\}$. 

If $|b_{i_0}|\geq |a_{i_0}|$, we have 
  $$
  y^a-y^b=y^{b-b_{i_0}}(y^{a_{i_0}}-y^{b_{i_0}})+y^{a_{i_0}}(y^{a-a_{i_0}}-y^{b-b_{i_0}}).
  $$ 
  Then we conclude as above.
\end{proof}
Set 
$$L_{0}^{\dagger}=\bigcup_{r>1}\left\{\sum_{u\in \mathbf{Z}^n }a_ux^u:a_u\in
K,\;|a_u|r^{|u|}\rightarrow 0~\as~|u|\rightarrow +\infty\right\}.$$
Note that $L_\delta^\dagger\subset L_0^\dagger$. Similarly, we can define the complex $C^{\cdot}(L_{0}^{\dagger},\underline{D})$ and the endomorphism $\alpha_1$. %
By \cite[Proposition 3.4]{berthelot1984cohomologie} or \cite{Bourgeois}, we have 
\begin{Proposition}\label{Compa-Basic2}
We have a canonical isomorphism
\begin{displaymath}
R\Gamma_{\rig}(\mathbf{T}^n_k/K,f^{*}\mathcal{L}_{-\psi})\xrightarrow{\sim } C^{\cdot}(L^{\dagger}_{0},\underline{D})
\end{displaymath} which is compatible with their endomorphisms $F_*=q^n(F^*)^{-1}$ and $\alpha_1$.
 Furthermore, $\alpha_1$ is bijective on each cohomology group of $C^{\cdot}(L^{\dagger}_{0},\underline{D})$.
\end{Proposition}
Let $\Sigma'$ be a regular refinement of $\Sigma$ and let $X'$ be the toric scheme over $R$ associated to $\Sigma'$. Denote by $X'_k,X'_K$ the special fiber and the generic fiber of $X'$ over $R$, respectively. Let $\mathbf{T}_R^n\rightarrow X'$ be the immersion of the maximal open dense torus, and let $D'=X'-\mathbf{T}_R^n$. Denote by $X'^{\an}_K$ (resp. $D'^{\an}_{K}$) the analytic space associated to $X'_K$ (resp. $D'_K$). Note that $X'$ is smooth and $D'$ is a normal crossing divisor in $X'$. Let $\Omega_{X'}^k(\log D')$ be the sheaf of differential forms of degree $k$ with logarithmic poles along $D'$. Then $\Omega_{X'^{\an}_K}^k(\log D'^{\an}_{K})$ is a free $\mathcal{O}_{X'^{\an}_K}$ module with basis $\frac{dx_{i_1}}{x_{i_1}}\wedge\cdots\wedge \frac{dx_{i_k}}{x_{i_k}}$, where $1\leq i_1<\cdots<i_k\leq n$ and $x_1,\cdots,x_n$ are coordinates of $\mathbf{T}_K^{n,\an}$. Consider the complex 
$$
E:=(\Omega'^{\cdot}_{X'^{\an}_K}(\log D'^{\an}_K),d)
$$ where the differential is given by $d=d_x+\sum_{i=1}^n \pi E_i\hat{f}\frac{dx_i}{x_i}$. Note that 
$$\hat{f}\in \Gamma(X^{\an}_K,\mathcal{O}_{X^{\an}_K})\cong \Gamma(X'^{\an}_K,\mathcal{O}_{X'^{\an}_K}).$$ Hence the twist de Rham complex $E$ is well defined. 
Let $\Sigma''$  be a regular fan containing $\Sigma'$ such that $|\Sigma''|=\mathbf{R}^n$, let $X''$ be the toric scheme over $R$ associated to $\Sigma''$ and let $\widehat{X''}$ be its formal completion. In \cite[Corollary A.4]{baldassarri2004dwork}, taking $\mathcal{X}=\widehat{X''}$ , $Z=X''-\mathbf{T}_R^n$, $V=X'_k$, $U=\mathbf{T}^n_k$ and $W=X'^{\an}_K$, we get a canonical isomorphism
\begin{equation}\label{comp-log}
    R\Gamma(X'^{\an}_K,j^{\dagger}\Omega^{\cdot}_{X'^{\an}_K}(\log D'^{\an}_K))
\cong R\Gamma_{\rig}(\mathbf{T}^n_k,\mathcal{L}_{-\psi,f}).
\end{equation}
Now we can give the proof of \textbf{Theorem} \ref{Thm1}.
\begin{proof}
Take cones $\sigma_1,\cdots,\sigma_m\in \Sigma'$ such that the open sets $U_i$ defined by $\sigma_i$ form an affine open covering of $X'$. Let $j_{\underline{i}}:=j_{i_1\cdots i_k}$ be the inclusion $U_{\underline{i}}:=U_{i_1}\cap\cdots \cap U_{i_k}\subset X'$. Note that $U_{\underline{i}}$ is the affine open set defined by the cone $\sigma_{\underline{i}}:=\sigma_{i_1}\cap \cdots \cap \sigma_{i_k}$. By \cite[Proposition 2.1.8]{berthelot1996cohomologie}, we have an exact sequence
$$
\xymatrix@R=12pt@C=8pt{
0\ar[r] & j^{\dagger}E\ar[r] & \underset{i}{\prod}j_i^{\dagger}E\ar[r]&\underset{i_1<i_2}{\prod}j_{i_1i_2}^{\dagger}E \ar[r]&\cdots\ar[r]&j_{1\cdots m}^{\dagger}E\ar[r]&0,
}
$$
where 
$
j^{\dagger}_{\underline{i}}E=\varinjlim_{V}j_{V*}j_{V}^{*}(E)
$ and $j_V:V\rightarrow X'^{\an}_K$ runs through strict neighborhoods of $]U_{\underline{i},k}[_{\widehat{\overline{X'}}}$ contained in
$X'^{\an}_K$. 
We still denote by $j_{\underline{i}}$ the inclusion $U^{\an}_{\underline{i},K}\subset X'^{\an}_K$. We have
$
j^{\dagger}_{\underline{i}}E=j_{\underline{i}*}j^{\dagger}(E|_{U^{\an}_{\underline{i},K}})
$ by \cite[1.2.4 (ii) and 2.1.1]{berthelot1996cohomologie}. 
Let $\check{\sigma_{\underline{i}}}$ be the dual of $\sigma_{\underline{i}}$ and $\sigma=\check{\delta}$. Applying Lemma \ref{Compa-Basic} to each $U^{\an}_{\underline{i},K}$, we know that
$R\Gamma(X'^{\an}_K,j^{\dagger}E)$ can be represented by the total complex of
\begin{equation}\label{total1}
\xymatrix@R=12pt@C=8pt{
& 0\ar[r] & \underset{i}{\prod}\Gamma(U^{\an}_{i,K},j^{\dagger}((E|_{U^{\an}_{i,K}})))\ar[r]&\cdots\ar[r]&\Gamma(U^{\an}_{1\cdots n,K},j^{\dagger}((E|_{U^{\an}_{1\cdots m,K}})))\ar[r]&0
}
\end{equation}
and $\Gamma(U^{\an}_{\underline{i},K},j^{\dagger}\mathcal{O}_{U^{\an}_{\underline{i},K}})\cong L^{\dagger}_{\check{\sigma_{\underline{i}}}}$, where
$$
L^{\dagger}_{\check{\sigma_{\underline{i}}}}=\bigcup_{\lambda>1}\left\{\sum_{u\in \check{\sigma_{\underline{i}}}\cap \mathbf{Z}^{n}}A_ux^{u}:\;A_{u}\in K,\;|A_u|\lambda^{|u|}\xrightarrow{|u|\rightarrow+\infty} 0\right\}
$$
which can be regarded as a subalgebra of $L_{0}^{\dagger}$. Thus (\ref{total1}) can be represented by the bicomplex 
\begin{equation}\label{total2}
\xymatrix@R=12pt@C=6pt{
& 0\ar[r]& \underset{i}{\prod}C^{\cdot}(L^{\dagger}_{\check{\sigma}_{i}})\ar[r]&\underset{i_1<i_2}{\prod}C^{\cdot}(L^{\dagger}_{\check{\sigma_{\underline{i}}}}) \ar[r]&\cdots\ar[r]&C^{\cdot}(L^{\dagger}_{\check{\sigma}_{i\dots m}})\ar[r]&0,
}
\end{equation} where all $C^{\cdot}(L_{\check{\sigma}_{\underline{i}}}^{\dagger})$ are regarded as subcomplexes of $C^{\cdot}(L_0^{\dagger},\underline{D})$. By Lemma \ref{dagger-ex} below and a spectral sequence argument, $C^{\cdot}(L_{\delta}^{\dagger},\underline{D})$ can be represented by the total complex of (\ref{total2}). Hence $$R\Gamma(X'^{\an}_K,j^{\dagger}E)\cong C^{\cdot}(L^{\dagger}_{\delta},\underline{D}).$$ By Proposition \ref{Compa-Basic2}, the isomorphism (\ref{comp-log}) can be represented by the inclusion
$$
C^{\cdot}(L^{\dagger}_{\delta},\underline{D})\hookrightarrow C^{\cdot}(L^{\dagger}_{0},\underline{D}).
$$ By (\ref{Commu-Dwork1}) and Proposition \ref{Compa-D}, we have 
$$
H_{n-i}(K_{\cdot}(L_{\delta}^{\dagger},\underline{\hat{D}}))\xrightarrow{\sim}H^{i}_{\rig}(\mathbf{T}^n_k/K,f^{*}\mathcal{L}_{-\psi})
$$
for each $i$. The second assertion of Theorem \ref{Thm1} follows from (\ref{Commu-Dwork1}) and Proposition \ref{Compa-Basic2}.
\end{proof}
\begin{Lemma}\label{exact-toric}
Notation as above, 
we have an exact sequence
\begin{equation}\label{exact-toric-nondagger}
\xymatrix@R=12pt@C=6pt{
& 0\ar[r] & L_{\check{\sigma},R}\ar[r] & \underset{i}{\prod}L_{\check{\sigma}_{i},R}\ar[r]&\underset{i_1<i_2}{\prod}L_{\check{\sigma_{\underline{i}}},R} \ar[r]&\cdots\ar[r]&L_{\check{\sigma}_{i\dots m},R}\ar[r]&0
,}
\end{equation}
where $L_{\check{\tau},R}=R[x^{\check{\tau}\cap \mathbf{Z}^n}]$ is regarded as a subalgebra of $R[x^{\mathbf{Z}^{n}}]$, for $\tau=\sigma,\sigma_{\underline{i}},\cdots$. The boundary map $\underset{i_1<\cdots <i_k}{\prod}L_{\check{\sigma_{\underline{i}}},R}\xrightarrow{d} \underset{i_1<\cdots<i_{k+1}}{\prod}L_{\check{\sigma_{\underline{i}}},R}$ is given by 
$$(dg)_{i_1\cdots i_{k+1}}=\sum_{s=1}^{k+1}(-1)^{s-1}g_{i_1\cdots \hat{i}_s\cdots i_{k+1}}$$ for any $g\in \underset{i_1<\cdots <i_k}{\prod}L_{\check{\sigma_{\underline{i}}},R}$.
\end{Lemma}
\begin{proof}
Consider the $\check{\textrm{C}}\textrm{ech}$ complex $C^{\cdot}(\mathfrak{U},\mathcal{O}_{X'})$ defined by the affine open covering $\mathfrak{U}=\left\{U_{\sigma_i}\right\}_{i=1}^m$. We have $\check{H}^{k}(\mathfrak{U},\mathcal{O}_{X'})\cong H^k(X',\mathcal{O}_{X'})$ for any $k$. Since $\phi:X'\rightarrow X$ is proper and $X$ is affine, we have
$R^k\phi_{*}\mathcal{O}_{X'}=H^{k}(X',\mathcal{O}_{X'})^{\sim}$. By \cite[Proposition 9.2.5]{cox}, we have $R^k\phi_{*}\mathcal{O}_{X'}=0$ for any $k\geq 1$. Thus (\ref{exact-toric-nondagger}) is exact except at the second position. Since $\sigma=\cup_{i=1}^{m} \sigma_i$, we have $\check{\sigma}=\cap_{i=1}^m \check{\sigma_i}$. Thus
$L_{\check{\sigma},R}=\textrm{Ker}(\underset{i}{\prod}L_{\check{\sigma}_{i},R}\rightarrow \underset{i_1<i_2}{\prod}L_{\check{\sigma_{\underline{i}}},R})$ and then (\ref{exact-toric-nondagger}) is exact.
\end{proof}

\begin{Lemma}\label{dagger-ex}
Notation as above, we have an exact sequence
\begin{equation}\label{equa-dagger-toric}
\xymatrix@R=12pt@C=6pt{
& 0\ar[r]&L_{\delta}^{\dagger}\ar[r]  & \underset{i}{\prod}L^{\dagger}_{\check{\sigma}_{i}}\ar[r]&\underset{i_1<i_2}{\prod}L^{\dagger}_{\check{\sigma_{\underline{i}}}} \ar[r]&\cdots\ar[r]&L^{\dagger}_{\check{\sigma}_{i\dots m}}\ar[r]&0
,}
\end{equation}where the boundary map is the same as (\ref{exact-toric-nondagger}).
\end{Lemma}
\begin{proof}
 By the definition of $\Sigma'$, we have $\sigma=\cup_{i=1}^m\sigma_i$. Thus $\delta=\check{\sigma}=\cap_{i=1}^m \check{\sigma_i}$ and $L^{\dagger}_{\check{\sigma}}=\textrm{Ker}(\underset{i}{\prod}L^{\dagger}_{\check{\sigma}_{i}}\rightarrow \underset{i_1<i_2}{\prod}L^{\dagger}_{\check{\sigma_{\underline{i}}}} )$.
Now consider $\underset{i_1<\cdots <i_k}{\prod}L^{\dagger}_{\check{\sigma_{\underline{i}}}}\xrightarrow{d} \underset{i_1<\cdots<i_{k+1}}{\prod}L^{\dagger}_{\check{\sigma_{\underline{i}}}}$, take $g=(g_{\underline{i}})\in \underset{i_1<\cdots <i_k}{\prod}L^{\dagger}_{\check{\sigma_{\underline{i}}}}$. Suppose that for all $\underline{i}$, $g_{\underline{i}}\in L''(b,c)$ for some $b>0$, $c\in \mathbf{R}$, where
\begin{displaymath}
L''(b,c)=\left\{\sum_{u}A_ux^u:\;A_u\in K,\; \ord A_u\geq b|u|+c\right\}.
\end{displaymath}
Since we always can choose some element $a\in R$ such that $ag_i\in L''(b,0)$ for all $i$, we treat the case $c=0$. Now $g_{\underline{i}}$ can be written as
$$
g_{\underline{i}}=\sum_{k=0}^{+\infty}\pi^{k}g_{\underline{i},k},~g_{\underline{i},k}\in R[x^{\check{\sigma}_{\underline{i}}\cap \mathbf{Z}^n}],
$$ where $g_{\underline{i},k}$ is a sum of monomials of degree lying in $[\frac{k}{b(p-1)},\frac{k+1}{b(p-1)})$. Note that the degree of polynomial defines a grading on the complex (\ref{exact-toric-nondagger}), and $d$ preservers the grading. Write $g=\sum_{k=0}^{+\infty}\pi^{k}g_k$, where $g_k=(g_{\underline{i},k})$. If $d(g)=0$, we have $d(g_k)=0$ for
each $k$. 
 By Lemma \ref{exact-toric}, there exists
$$
h_k=(h_{\underline{i},k})_{\underline{i}}\in \underset{i_1<\cdots <i_{k-1}}{\prod}L_{\check{\sigma_{\underline{i}}},R}$$ such that $d(h_k)=g_k$ for each $k$ and we may assume that each $h_{\underline{i},k}$ is a sum of monomials of degree lying in $[\frac{k}{b(p-1)},\frac{k+1}{b(p-1)})$.
 Suppose that $h_{\underline{i},k}=\sum_{u}A_{\underline{i},k,u}x^u$, we have 
$$\ord(\pi^{k}A_{\underline{i},k,u})\geq \frac{k}{p-1}> \frac{k}{(p-1)\frac{k+1}{b(p-1)}}|u|\geq \frac{b}{2}|u|$$ for any $k\geq 1$. Thus
$\sum_{k=1}^{+\infty}\pi^{k}h_{\underline{i},k}\in L''(\frac{b}{2},0)$. Note that $h_{\underline{i},0}$ is a polynomial for each $\underline{i}$.  We have 
$$
h:=\sum_{k=0}^{+\infty}\pi^{k}h_k\in \underset{i_1<\cdots <i_{k-1}}{\prod}L^{\dagger}_{\check{\sigma_{\underline{i}}}}
$$ and $d(h)=g$. Hence (\ref{equa-dagger-toric}) is exact.
\end{proof}

\section{Applications}
Define 
$$B=\left\{\sum\limits_{u\in \delta\cap\mathbf{Z}^n }a_u\widetilde{\gamma}^{Mw(u)}x^u:\;a_u\in K,\;a_u \rightarrow 0 ~\as ~w(u)\rightarrow +\infty\right\}.$$
We can define the 
Koszul complexes $K_{\cdot}(L(b),\underline{\hat{D}})$ $(0<b\leq \frac{p}{p-1})$ and $K_{\cdot}(B,\underline{\hat{D}})$ with actions of $\alpha$, respectively. They can be regarded as subcomplexes of $K_{\cdot}(L_{\delta}^{\dagger},\underline{\hat{D}})$.  
\begin{Proposition}\label{Dwork-surj}
For each $i$, the inclusions $B\hookrightarrow L(\frac{1}{p-1})\hookrightarrow L^{\dagger}_{\delta}$ induce a surjective homomorphism $H_{i}(K_{\cdot}(B,\underline{\hat{D}}))\rightarrow H_{i}(K_{\cdot}(L_{\delta}^{\dagger},\underline{\hat{D}})).$
\end{Proposition}
\begin{proof}
We use the same method as \cite[Corollaire 1.3]{Bourgeois}. 
The inclusions $L(\frac{p}{p-1})\hookrightarrow B\hookrightarrow L_{\delta}^{\dagger}$ imply that it suffices to prove the homomorphism $H_{i}(K_{\cdot}(L(\frac{p}{p-1}),\underline{\hat{D}}))\rightarrow H_{i}(K_{\cdot}(L_{\delta}^{\dagger},\underline{\hat{D}}))$ is surjective.
Note that $L_{\delta}^{\dagger}=\bigcup_{0<b\leq \frac{p}{p-1}}L(b)$. By Theorem \ref{Thm1} and \cite[3.10]{berthelot1997finitude}, $H_{i}(K_{\cdot}(L_{\delta}^{\dagger},\underline{\hat{D}}))$ is finite dimensional. Thus there exists some $0<b\leq \frac{p}{p-1}$ such that
$
H_{i}(K_{\cdot}(L(b),\underline{\hat{D}}))\rightarrow H_{i}(K_{\cdot}(L_{\delta}^{\dagger},\underline{\hat{D}}))
$ is surjective. Note that $\psi_q(L(b))\subset L(qb)$. 
Consider the following composite
$$
L(b)\xrightarrow{\cdot F_0}L\Big(\min\left\{\frac{p}{q(p-1)},b\right\}\Big)\xrightarrow{\psi_q}L\Big(\min\left\{\frac{p}{p-1},qb\right\}\Big).
$$ Hence
$$\alpha(L(b))\subset L\Big(\min\left\{\frac{p}{p-1},qb\right\}\Big).$$ So there exists a positive integer $m$ such that
$$
\alpha^{m}(H_{i}(K_{\cdot}(L(b),\underline{\hat{D}})))\subset H_{i}\Big(K_{\cdot}(L(\frac{p}{p-1}),\underline{\hat{D}})\Big).$$ 
We have a commutative diagram
$$
\xymatrix{
& H_{i}(K_{\cdot}(L(b),\underline{\hat{D}})) \ar[d]^{\alpha^m} \ar[r] &  H_{i}(K_{\cdot}(L_{\delta}^{\dagger},\underline{\hat{D}}))\ar[d]^{\alpha^m}\\
& H_{i}(K_{\cdot}(L(\frac{p}{p-1}),\underline{\hat{D}})) \ar[r] & H_{i}(K_{\cdot}(L_{\delta}^{\dagger},\underline{\hat{D}}))
.}
$$
The top horizontal arrow is surjective by the choice of $b$. 
The right vertical arrow $\alpha^m$ is bijective by Theorem \ref{Thm1}. So the bottom horizontal arrow is surjective.
\end{proof}
Let $X=\mathbf{T}_k^r\times \mathbf{A}_k^{n-r}$ and
$$f\in k[x_1,\cdots,x_n,(x_1\cdots x_r)^{-1}].$$ For any subset $A\subset S_r=\left\{r+1,\cdots,n\right\}$, let $$\mathbf{R}_A^n=\left\{(x_1,\cdots,x_n):\;x_i=0~\textrm{for}~i\in A\right\}.$$ Let $V_A(f)$ be the volume of $\Delta(f)\cap \mathbf{R}_A^n$ with respect to Lebesgue measure on $\mathbf{R}_A^n$. Set 
$$v_{A}(f)=\sum_{B\subset A}(-1)^{|B|}(n-|B|)!V_B(f)$$
and $\tilde{r}= \dim \Delta(f_{S_r})$. We say that $f$ is \textit{commode with respect to} $S_r$ if for all subsets $A\subset S_r$, we have $\dim\Delta(f_A)=\tilde{r}+|S_r-A|,
$ where $f_A:=\prod_{j\in A} \theta_{j}\circ f(x)$ and $\theta_j$ is the map such that $x_j\mapsto 0$.
\begin{Theorem}\label{Calculate-rig}
Let $X=\mathbf{T}_k^r\times \mathbf{A}_k^{n-r}$, $f\in \Gamma(X,\mathcal{O}_X)$. Suppose that $f$ is commode with respect to $S_r$, nondegenerate with respect to $\Delta(f)$ and $\dim\Delta(f)=n$. Then

(i) $H_{c,\rig}^i(X,f^{*}\mathcal{L}_{\psi})=H_{\rig}^i(X,f^{*}\mathcal{L}_{-\psi})=0$ if $i\neq n$.

(ii) $\dim H_{c,\rig}^n(X,f^{*}\mathcal{L}_{\psi})=\dim H_{\rig}^n(X,f^{*}\mathcal{L}_{-\psi})=v_{S_r}(f)$.
\end{Theorem}
\begin{proof}
First we consider the case $r=n$. Since $X$ is affine, smooth and pure of dimension $n$, by the Poincar\'{e} duality and the trace formula in the rigid cohomology theory, we have
\begin{equation}\label{L-fun}
    L(\mathbf{T}^n_k,f,t)=\prod_{i=0}^{n}\det(I-tq^n(F^{*})^{-1}|H_{\rig}^i(\mathbf{T}^n_k/K,f^{*}\mathcal{L}_{-\psi})^{(-1)^{i+1}}. 
\end{equation}
When $f$ is nondegenerate, we have $H_{i}(K_{\cdot}(B,\underline{\hat{D}}))=0$ for $i\neq 0$ and $\dim H_{0}(K_{\cdot}(B,\underline{\hat{D}}))=n!\Vol(\Delta(f))$ by \cite[Theorem 2.18]{Ado1989}. By Proposition \ref{Dwork-surj}, we have
$
H_{i}(K_{\cdot}(L^{\dagger}_{\delta},\underline{\hat{D}}))=0
$ for $i\neq 0$. By Theorem \ref{Thm1}, we have 
$$H^{i}_{\rig}(\mathbf{T}_k^n,f^{*}\mathcal{L}_{-\psi})\cong H_{n-i}(K_{\cdot}(L^{\dagger}_{\delta},\underline{\hat{D}}))=0
$$ for $i\neq n$. By \cite[(2.5)]{Ado1989} and (\ref{L-fun}), we have 
$$
\det(I-tq^n(F^{*})^{-1}|H_{\rig}^n(X/K,f^{*}\mathcal{L}_{-\psi}))=\det(1-tH_{0}(\alpha)|H_{0}(K_{\cdot}(B,\underline{\hat{D}}))).
$$ So $$
\dim H_{\rig}^n(X/K,f^{*}\mathcal{L}_{-\psi})=\dim H_{0}(K_{\cdot}(B,\underline{\hat{D}}))=n!\Vol(\Delta(f)).
$$ We then use induction on $n-r$. Suppose that Theorem \ref{Calculate-rig} holds for $n-(r+1)$. 
By \cite[2.3.1]{berthelot1997finitude} and \cite[Theorem 4.1.1]{tsuzuki1999gysin}, we have a distinguished triangle
\begin{displaymath}
R\Gamma_{\rig}(X,\mathcal{L}_{-\psi,f})\rightarrow R\Gamma_{\rig}(U,\mathcal{L}_{-\psi,f})\rightarrow R\Gamma_{\rig}(Z,\mathcal{L}_{-\psi,f_{r+1}})[-1]\rightarrow
\end{displaymath}
where 
$U=X-V(x_{r+1})\cong \mathbf{T}_{k}^{r+1}\times \mathbf{A}_k^{n-r-1}$ and $Z=V(x_{r+1})\cong \mathbf{T}_{k}^{r}\times \mathbf{A}_k^{n-r-1}$. One can check that $f$, $f_{r+1}$ are still commode with respect to $S_{r+1}$, and $f_{r+1}$ is nondegenerate with respect to $\Delta(f_{r+1})$. By the induction hypothesis, (i) holds and
\begin{eqnarray*}
    \dim H^n_{\rig}(X,\mathcal{L}_{-\psi,f})&=&\dim H^n_{\rig}(U,\mathcal{L}_{-\psi,f})-\dim H^{n-1}_{\rig}(Z,\mathcal{L}_{-\psi,f_{r+1}})
    \\&=&v_{S_{r+1}}(f)-v_{S_{r+1}}(f_{r+1})=v_{S_r}(f).
\end{eqnarray*}
\end{proof}



\end{document}